\documentclass[12pt,reqno]{amsart}
\usepackage{amssymb,amsmath,amscd}
\usepackage{mathrsfs,leftidx}
\input xy
\xyoption{all}
\entrymodifiers={+!!<0pt,\fontdimen22\textfont2>}

\def\RHom{\operatorname{RHom}}
\def\Hom{\operatorname{Hom}}
\def\op{\operatorname{op}}
\def\End{\operatorname{End}} 
\def\Ext{\operatorname{Ext}}
\def\EssIm{\operatorname{Ess.\!Im}}

\def\Mod{\operatorname{Mod}}

\def\H{\operatorname{H}}
\def\Ab{\operatorname{Ab}}

\newtheorem{Lemma}{Lemma}[section]
\newtheorem{Theorem}[Lemma]{Theorem}
\newtheorem*{Theorem*}{Theorem}
\newtheorem{Proposition}[Lemma]{Proposition}
\newtheorem{Corollary}[Lemma]{Corollary}
\newtheorem*{Corollary*}{Corollary}
\theoremstyle{definition}
\newtheorem{Definition}[Lemma]{Definition}

\newtheorem{Remark}[Lemma]{Remark}

\newtheorem{DGA's}[Lemma]{DGA's}

\begin{document}

\setlength{\parindent}{0pt}
\setlength{\parskip}{7pt}

\title[Derived Equivalences of Upper Triangular DGA's]
{Derived Equivalences of Upper Triangular Differential Graded Algebras}
\author{Daniel Maycock}

\maketitle
\section{Introduction}
The question of when two derived categories of rings are equivalent has been studied extensively. Morita theory answered the question of when two module categories of rings are equivalent and a version of Morita theory for derived categories was developed by Rickard in \cite{Rik} which made use of the concept of tilting modules. This approach was applied by Ladkani in \cite{Lad} to the situation of derived equivalences of upper triangular matrix rings. In this paper we extend the main results from \cite{Lad} to the more general case of upper triangular matrix differential graded algebras (henceforth referred to as DGAs). For this we will make extensive use of the tool of recollements and in particular the situation given by J\o rgensen in \cite{Jor}.

Section 2 sets out the notation used. We begin properly in section 3 by introducing the upper triangular matrix DGA $\Lambda$ which has the form $\begin{bmatrix}R&M\\0&S\end{bmatrix}$, where $R$ and $S$ are DGAs and $\leftidx{_R}{M}{_S}$ is a $R$-$S$-DG-bimodule. There are also the left-DG-modules $B=\begin{bmatrix}R\\0\end{bmatrix}$ and $C=\begin{bmatrix}M\\S\end{bmatrix}$ which we will use throughout the paper. Next, by proving some properties of $B$ and $C$, we are able to use the main result from \cite{Jor} to obtain the recollement
$$\xymatrix@C5pc{D(R) \ar[r]^{i_*} & D(\Lambda) \ar@/^2pc/[l]^{i^!} \ar@/_2pc/[l]_{i^*} \ar[r]^{j^*} & D(S) \ar@/^2pc/[l]^{j_*} \ar@/_2pc/[l]_{j_!}}$$
where $i_*(_RR)\cong B$ and $j_!(_SS)\cong C$.
We then conclude the section by presenting some useful results obtained from the recollement which we will need in the next section. 

In section 4 we turn our attention to the main aim of the paper, generalizing the main theorem from Ladkani to DGA's. To do so we follow a similar idea as used in the proof of \cite[Theorem 4.5]{Lad}, by considering the DG-module $T=\Sigma i_*X\oplus j_*j^*\Lambda$ where $X$ is compact and $\langle X\rangle=D(R)$, where $\langle X\rangle$ denotes the smallest triangular subcategory containing $X$ which is closed under the taking of coproducts. We begin with a statement of Keller's theorem which we will require to prove the following theorem, our ``first attempt'' at generalising \cite[Theorem 4.5]{Lad}. 
\begin{Theorem*}
Let $X$ be a DG $R$-module such that $_RX$ is compact and $\left\langle _RX\right\rangle=D(R)$. Let $_RM_S$ be compact as a DG-$R$-module. Let $T=\Sigma i_*X\oplus j_*j^*\Lambda$ with $\mathscr{E}=\End_\Lambda(P)$, where $P$ is a K-projective resolution of $T$. Then $\mathscr{E}$ is an DGA with $D(\Lambda)\simeq D(\mathscr{E}^{\op})$.
\end{Theorem*}
We then turn our attention to considering $P$, the K-projective resolution of $T$, and by doing so we are able to calculate its endomorphism DGA, which leads to our generalisation of the main theorem of Ladkani below.
\begin{Theorem*}
Let $X$ be a DG $R$-module such that $_RX$ is compact and $\left\langle _RX\right\rangle=D(R)$. Let $_RM_S$ be compact as an DG $R$-module and let $U$ and $V$ be K-projective resolutions of $X$ and $M$ respectively. 
Then for the upper triangular differential graded algebras  
$$\Lambda=\begin{bmatrix}R&M\\0&S\end{bmatrix}\textrm{ and } \tilde{\Lambda}=\begin{bmatrix}S&\Hom_R(V,U)\\0&\Hom_R(U,U)^{\op}\end{bmatrix}$$
we have that $D(\Lambda)\simeq D(\tilde{\Lambda})$.
\end{Theorem*} 

One specific advantage of considering the DGA case rather than the ring case is that with the DGA case we can do without a lot of constraints which are required in the ring case to ensure that the derived equivalence is between two triangular matrix rings. 
  
Finally in section 5 we conclude with a look at some special cases. In the first we reconsider the original case in Ladkani, involving just rings and show that by making the same assumptions in our general theorem we obtain the same equivalence.

We then briefly consider what happens in the special case where $\leftidx{_R}{X}=\leftidx{_R}{R}$. In the final example, we require that our DGA's are over some field $k$ and that $R$ is self dual, that is, $\Hom_k(R,k)\cong R$ in the derived category of DG-R-modules. This gives us the following result. 
\begin{Corollary*}
Let $R$ be a self dual finite dimensional DGA and $S$ be a DGA, both over a field $k$. Let $\leftidx{_R}{M}{_S}$ be compact as a DG-$R$-module. Then 
$$\Lambda=\begin{bmatrix}R&M\\0&S\end{bmatrix} \textrm{ and } \tilde{\Lambda}=\begin{bmatrix}S&DM\\0&R\end{bmatrix}$$
are derived equivalent.
\end{Corollary*}

\section{Notation and terminology}
In this we will fix the notation which we shall use throughout the paper; more details on DGA's and DG-modules can be found in \cite{Kel} and \cite{Fra}.

Throughout this paper we will make use of Differential Graded Algebras; these are always assumed to be over some commutative ground ring $k$ unless stated otherwise.  For any graded object $r$ we will denote its degree by $|r|$.  Note that we shall observe the Koszul sign convention so that whenever two graded elements of degrees $m$ and $n$ are interchanged we introduce a sign $(-1)^{mn}$. 

For a DGA $R$ we can define the opposite DGA, denoted $R^{\op}$, this is the same as $R$ except that the product is given by $r.s=(-1)^{|r||s|}sr$ where . denotes multiplication in $R^{\op}$. We will often identify DG-right-$R$-modules with DG-left-$R^{\op}$-modules.

We shall often need to consider DG-modules with more than one DG-module structure, for instance a DG-left-$R$-right-$S$-module, denoted by $\leftidx{_R}{M}{_S}$. In these cases the different structures are required to be compatible, for the $\leftidx{_R}{M}{_S}$ case this means that the rule $(rm)s=r(ms)$ holds.

\medskip

For a DGA $R$ we denote the category of all DG-left-$R$-modules by Mod$\:R$. We define the homotopy category of $R$, which we denote by $K(R)$, as the category consisting of all DG-left-$R$-modules whose morphisms are the morphisms of DG-modules mudololo homotopy. We define the derived category of $R$, denoted by $D(R)$, from $K(R)$ by formally inverting the quasi-isomorphisms. Both $K(R)$ and $D(R)$ are triangulated categories. A more detailed construction of the derived category and details of triangulated categories can be found in \cite{Har}. 

Since we can identify DG-right-$R$-modules with DG-left-$R^{\op}$-modules we can also identify the derived category of DG-right-$R$-modules with $D(R^{\op})$.

\bigskip

\section{A Recollement Situation}
We begin by defining the differential graded algebras and DG-modules which we will be using throughout the paper.

\medskip

\begin{Definition}
\label{UT DG def}
Throughout this paper, let $R$ and $S$ be Differential Graded algebras with $\leftidx{_R}{M}{_S}$ a DG-bimodule which is quasi-isomorphic to $\leftidx{_R}{V}{_S}$ where $V$ is K-projective as a DG-left-$R$-module and let $\Lambda=\begin{bmatrix}R & M\\0 & S	\end{bmatrix}$ denote the upper triangular matrix DGA with the differential $\partial^\Lambda\begin{bmatrix}r & m\\0 & s	\end{bmatrix}
=\begin{bmatrix}\partial^R r & \partial^M m\\0 & \partial^S	s\end{bmatrix}$.
\end{Definition}

\bigskip

\begin{Remark}
In the case where the base ring $k$ is a field we always have that $\leftidx{_R}{M}{_S}$ is quasi-isomorphic to some $\leftidx{_R}{V}{_S}$ with $V$ K-projective as a DG-$R$-module.
\end{Remark}

\bigskip

\begin{Definition}
Let $e_R=\begin{bmatrix} 1 & 0\\0 & 0\end{bmatrix}$ and $e_S=\begin{bmatrix} 0 & 0\\0 & 1\end{bmatrix}$
and define the DG-left-$\Lambda$-modules $$B=\Lambda e_R=\begin{bmatrix} R\\0\end{bmatrix}\textrm{ and } C=\Lambda e_S=\begin{bmatrix} M\\S\end{bmatrix}$$
where $B$ has the differential $\partial\left(\begin{bmatrix}r\\0\end{bmatrix}\right)=\left(\begin{bmatrix}\partial^Rr\\0\end{bmatrix}\right)$ 
and $C$ has the differential 
$\partial\left(\begin{bmatrix}m\\s\end{bmatrix}\right)=\left(\begin{bmatrix}\partial^Mm\\\partial^Ss\end{bmatrix}\right)$.
\end{Definition}

\bigskip

The first aim is to construct a recollement involving the objects we have defined above. To do this we shall use \cite[Theorem 3.3]{Jor} but before we can use this theorem we first need the following Lemmas involving the DG-modules $B$ and $C$. 

\bigskip

\begin{Definition}
For $X$ a full subcategory of a triangulated category $T$, we can define a full subcategory  
$$X^\bot=\{Y\in T\:|\:\Hom_T(\Sigma^l X,Y)=0 \textrm{ for all } l\}$$
\end{Definition}

\bigskip 

\begin{Lemma} 
$\Lambda\cong B\oplus C$ in $D(\Lambda)$ and hence both $B$ and $C$ are K-projective DG-left-$\Lambda$-modules which are compact in $D(\Lambda)$. 
\end{Lemma}

\begin{proof}
Define $\Theta : \Lambda \rightarrow B\oplus C$ and $\Phi : B\oplus C \rightarrow \Lambda$ by 
$$\Theta\left(\begin{bmatrix}r & m\\0 & s	\end{bmatrix}\right) = 
\left(\begin{bmatrix} r\\0\end{bmatrix},\begin{bmatrix} m\\s\end{bmatrix}\right)$$ 
and $$\Phi \left(\left(\begin{bmatrix} r\\0\end{bmatrix},\begin{bmatrix} m\\s\end{bmatrix}\right)\right) = \begin{bmatrix}r & m\\0 & s	\end{bmatrix}.$$ 

\smallskip

It is obvious that $\Theta$ and$\Phi$ are inverses of each other and it is straightforward to check that they are homomorphisms of DG-modules.
So we have that $\Lambda\cong B\oplus C$ as DG-$\Lambda$-modules and hence also in $D(\Lambda)$.
\end{proof}

\bigskip

\begin{Lemma} 
$B\in C^\bot$ as DG-modules and hence in $D(\Lambda)$.
\end{Lemma}

\begin{proof} 
Let $C \stackrel{f}{\rightarrow} B$ be a morphism of DG-modules. It suffices to show that $f=0$ and since $\begin{bmatrix}0\\1\end{bmatrix}$ generates $C$ we only need to show that $f\left(\begin{bmatrix}0\\1\end{bmatrix}\right)=0$. 

Let $f\left(\begin{bmatrix}0\\1\end{bmatrix}\right)=\begin{bmatrix}r\\0\end{bmatrix}$ for some $r\in R$. Then $$f\left(\begin{bmatrix}0\\1\end{bmatrix}\right)=f\left(e_S.\begin{bmatrix}0\\1\end{bmatrix}\right)=e_Sf\left(\begin{bmatrix}0\\1\end{bmatrix}\right)= \begin{bmatrix}0&0\\0&1\end{bmatrix}\begin{bmatrix}r\\0\end{bmatrix}=0$$ as required, hence $f=0$ and so $B\in C^\bot$.   
\end{proof}

\bigskip

\begin{Lemma}
$B^\bot\cap C^\bot =0$ in $D(\Lambda)$.
\end{Lemma}

\begin{proof} 
Let $X\in B^\bot\cap C^\bot$, then $\Hom_{D(\Lambda)}(\Sigma^iB,X)=0$ and \\$\Hom_{D(\Lambda)}(\Sigma^iC,X)=0$ for each $i$. 
$$H^iX\cong H^i\Hom_\Lambda(\Lambda,X)\cong \Hom_{K(\Lambda)}(\Lambda,\Sigma^iX)$$
$$\cong \Hom_{D(\Lambda)}(\Lambda,\Sigma^iX) \cong \Hom_{D(\Lambda)}(B\oplus C,\Sigma^iX)$$
$$\cong \Hom_{D(\Lambda)}(B,\Sigma^iX)\oplus \Hom_{D(\Lambda)}(C,\Sigma^iX)$$ 
$$\cong 0\oplus 0=0$$
for all $i$. Hence we have that $X\cong 0$ in $D(\Lambda)$ and so $B^\bot\cap C^\bot =0$. 
\end{proof}

\bigskip

We now have shown that $B$ and $C$ satisfy the conditions required to apply \cite[Theorem 3.3]{Jor}. However before we do so we prove the following lemma about the endomorphism DGAs of $B$ and $C$.

\begin{Lemma} 
Let $\mathscr{F}=\End_\Lambda(B)$ and $\mathscr{G}=\End_\Lambda(C)$, then $\mathscr{F}^{\op}\cong R$ and $\mathscr{G}^{\op}\cong S$ as Differential Graded Algebras.
\end{Lemma}

\begin{proof} 
Since $\begin{bmatrix}1\\0\end{bmatrix}$ is a generator of $B$ each element of $\End_\Lambda(B)$ depends entirely on where it sends $\begin{bmatrix}1\\0\end{bmatrix}$. For each $r\in R$ define the homomorphism $f_r$ as the element of $\mathscr{F}$ which sends $\begin{bmatrix}1\\0\end{bmatrix}$ to 
$\begin{bmatrix}r\\0\end{bmatrix}$.

We can now define $\phi:R^{\op} \rightarrow \mathscr{F}$ by $\phi(r)=f_r$. Since elements of $\mathscr{F}$ depend entirely on where they send $\begin{bmatrix}1\\0\end{bmatrix}$ this is obviously a bijection. It is also straightforward to show that $\phi$ is a homomorphism and so an isomorphism of DGAs.

\medskip

Now let $g\in\mathscr{G}$. Since $C$ is generated by $\begin{bmatrix}0\\1\end{bmatrix}$ we know that $g$ depends entirely on where it sends $\begin{bmatrix}0\\1\end{bmatrix}$. Let $g\left(\begin{bmatrix}0\\1\end{bmatrix}\right)=\begin{bmatrix}m\\s\end{bmatrix}.$ 

However $g\left(e_S.\begin{bmatrix}0\\1\end{bmatrix}\right)=g\left(\begin{bmatrix}0\\1\end{bmatrix}\right)=e_Sg\left(\begin{bmatrix}0\\1\end{bmatrix}\right)=e_S\begin{bmatrix}m\\s\end{bmatrix}
=\begin{bmatrix}0\\s\end{bmatrix}$, so $m=0$ and hence $g\left(\begin{bmatrix}0\\1\end{bmatrix}\right)=\begin{bmatrix}0\\s\end{bmatrix}$. 
So for each $s\in S$ we can define the homomorphism $g_s\in\mathscr{G}$ as the element of $\mathscr{G}$ which sends $\begin{bmatrix}0\\1\end{bmatrix}$ to $\begin{bmatrix}0\\s\end{bmatrix}.$

Hence we can define a map $\theta:S^{\op} \rightarrow \mathscr{G}$ sending $s\mapsto g_s$ which is easily shown to be an isomorphism and so $S^{\op}\cong\mathscr{G}$.  
\end{proof}
   
\bigskip

By taking the above lemmas together with \cite[Theorem 3.3]{Jor} we get the following recollement:

$$\xymatrix@C5pc{D(R) \ar[r]^{i_*} & D(\Lambda) \ar@/^2pc/[l]^{i^!} \ar@/_2pc/[l]_{i^*} \ar[r]^{j^*} & D(S) \ar@/^2pc/[l]^{j_*} \ar@/_2pc/[l]_{j_!}}.$$

Five of the functors are given by 
$$\begin{array}{ll}
&j_!(\_)=\leftidx{_\Lambda}{C}{_S}\stackrel{L}{\otimes}_S\_,\\
i_*(\_)=\leftidx{_\Lambda}{B}{_R}\stackrel{L}{\otimes}_R\_,&j^*(\_)=\RHom_\Lambda(_{\Lambda}C_S,\_),\\
i^!(\_)=\RHom_\Lambda(_{\Lambda}B_R,\_),&j_*(\_)=\RHom_S(_SC^*_\Lambda,\_).
\end{array}$$

\medskip

Here $_SC^*_\Lambda =\RHom_\Lambda (_\Lambda C_S,\Lambda)$.

\medskip

In particular, $i_*(R)\cong B$ and $j_!(S)\cong C.$

\bigskip

We shall now end this section with a number of results, involving the recollement we have constructed, which we will find to be of great use in the next section. 

\bigskip

\begin{Remark}
The functor $i_*(-)=\leftidx{_\Lambda}{B}{_R}\stackrel{L}{\otimes}_R\_$ sends a DG-$R$-module $X$ to the DG-$\Lambda$-module $\begin{bmatrix}X\\0\end{bmatrix}$. 
\end{Remark}

\bigskip

\begin{Proposition}
\label{Lem: C* K-proj}
For $_SC^*_\Lambda =\RHom_\Lambda (_\Lambda C_S,\Lambda)$ we have that:
\begin{enumerate}

	\item $C^*\cong _S\!\begin{bmatrix}0 & S\end{bmatrix}_\Lambda$ as DG-left-$S$-right-$\Lambda$-modules where $\begin{bmatrix}0&S\end{bmatrix}$ has the differential 
	$$\partial^{\left[\begin{smallmatrix}0 & S\end{smallmatrix}\right]}\left(\begin{bmatrix}0 & s\end{bmatrix}\right)
	=\begin{bmatrix}0 & \partial^Ss \end{bmatrix}.$$ 
	
	\medskip
	
	\item $\leftidx{_S}{C}{_\Lambda^*}$ is a K-projective object over both $S$ and $\Lambda$.
\end{enumerate}

\end{Proposition} 

\begin{proof}
(i) First observe that $C$ is generated by $\begin{bmatrix}  0\\ 1\end{bmatrix}$ and that $$\RHom_\Lambda (_\Lambda C_S,\Lambda)\simeq \Hom_\Lambda (_\Lambda C_S,\Lambda)$$ since $C$ is K-projective over $\Lambda$.

Let $\theta\in\Hom_\Lambda (C,\Lambda)$ such that $\theta\left(\begin{bmatrix}  0\\ 1\end{bmatrix}\right)= 
\begin{bmatrix}r & m\\0 & s	\end{bmatrix} \in \Lambda$. 

However $$\theta\left(\begin{bmatrix}  0\\ 1\end{bmatrix}\right)= \theta\left(e_S.\begin{bmatrix}  0\\ 1\end{bmatrix}\right)=e_S.\theta\left(\begin{bmatrix}  0\\ 1\end{bmatrix}\right)=e_S\begin{bmatrix}r & m\\0 & s	\end{bmatrix}=\begin{bmatrix}0 & 0\\0 & s	\end{bmatrix}.$$

So $\theta\left(\begin{bmatrix}  0\\ 1\end{bmatrix}\right)= 
\begin{bmatrix}0 & 0\\0 & s	\end{bmatrix}$.

So for every $s\in S$ we can define $\theta_s\in\Hom_\Lambda(C,\Lambda)$ to be the element which sends $\begin{bmatrix}  0\\ 1\end{bmatrix}$ to $\begin{bmatrix}0 & 0\\0 & s	\end{bmatrix}$.

We can now use this to define an map $\Theta:\Hom_\Lambda(C,\Lambda)\rightarrow \begin{bmatrix} 0 & S \end{bmatrix}$ given by $\Theta(\theta_s)=\begin{bmatrix} 0 & s \end{bmatrix}$. This map is obviously a bijection and it is straightforward to check that it is an isomorphism of DG-left-S-right-$\Lambda$-modules.

\medskip

(ii) To see that $C^*$ is K-projective over $S$ we observe that $C^*\cong S$ as $S$-modules. It remains to show now that $C^*$ is also K-projective over $\Lambda$. We do this by showing that $C^*\cong\begin{bmatrix} 0 & S \end{bmatrix}$ is a direct summand of $\Lambda$ as DG-right-$\Lambda$-modules.

\smallskip

First observe that $\begin{bmatrix} R & M \end{bmatrix}$ is a DG-right-$\Lambda$-module with the differential $\partial^ {\left[\begin{smallmatrix} r & m \end{smallmatrix}\right]}= \begin{bmatrix} \partial^Rr & \partial^Mm \end{bmatrix}$. 

\smallskip

Now define $\Phi:\Lambda_\Lambda\rightarrow\begin{bmatrix} R & M \end{bmatrix}_\Lambda\oplus\begin{bmatrix} 0 & S \end{bmatrix}_\Lambda$ such that $$\Phi\left(\begin{bmatrix} r & m \\ 0 & s \end{bmatrix}\right)=\left(\begin{bmatrix} r & m \end{bmatrix},\begin{bmatrix} 0 & s \end{bmatrix}\right).$$

It is clear to see that $\Phi$ is bijective and a homomorphism of DG-modules. So $C^*_\Lambda\cong\begin{bmatrix} 0 & S \end{bmatrix}_\Lambda$ is a direct summand of $\Lambda_\Lambda$ and so is a K-projective DG-right-$\Lambda$-module.

\end{proof}

\bigskip

\begin{Lemma}
\label{useful facts}
In the set up of the recollement we have that:
\begin{enumerate} 
	\item $j^*(\Lambda)\cong\leftidx{_S}{S}$ in $D(S)$,
	
	\medskip
	
	\item $j_*(_SS)\cong  \dfrac{\leftidx{_\Lambda}{C}}{\leftidx{_\Lambda}{\left[\begin{smallmatrix} M\\0\end{smallmatrix}\right]}{}}$ in $D(\Lambda)$.
\end{enumerate}
\end{Lemma}
\begin{proof}
(i) $j^*(\Lambda)=\RHom_\Lambda(\leftidx{_\Lambda}{C}{_S},\Lambda)=\leftidx{_S}{C}{^*}\cong\leftidx{_S}{S}$.
	
\medskip
	
(ii) Since $C^*$ is a K-projective $S$-module we have that
$$j_*(_SS)=\RHom_S(_SC^*_\Lambda,_SS)\cong \Hom_S(_SC^*_\Lambda,_SS)\cong \Hom_S(_S\begin{bmatrix} 0 & S \end{bmatrix}_\Lambda,_SS).$$

Now observe that $\leftidx{_S}{\begin{bmatrix} 0 & S \end{bmatrix}}$ is generated by $\begin{bmatrix} 0 & 1 \end{bmatrix}$ and for all $s\in S$ define $\phi_s \in \Hom_S(\leftidx{_S}{C}{^*_\Lambda},\leftidx{_S}{S})$ to be the element which sends $\begin{bmatrix} 0 & 1 \end{bmatrix}$ to $s$. We can now define the map $\Phi : \Hom_S(\leftidx{_S}{C}{^*_\Lambda},\leftidx{_S}{S}) \rightarrow \leftidx{_\Lambda}{C}/\leftidx{_\Lambda}{\left[\begin{smallmatrix} M\\0\end{smallmatrix}\right]}$ given by $\Phi(\phi_s)=\overline{\begin{bmatrix} 0\\s\end{bmatrix}}$, where $\overline{\begin{bmatrix} 0\\s\end{bmatrix}}$ denotes the element $\begin{bmatrix} 0\\s\end{bmatrix}+\begin{bmatrix}M\\0\end{bmatrix}$ in $\leftidx{_\Lambda}{C}/\leftidx{_\Lambda}{\left[\begin{smallmatrix} M\\0\end{smallmatrix}\right]}$. This is obviously a bijection and it is easy to show that it is a homomorphism of DG-$\Lambda$-modules.
\end{proof}
\bigskip

\section{Derived Equivalences of Upper Triangular DGA's}

We are now almost in the position where we can make a start on what is the main aim of the paper, to obtain a generalised version of \cite[Theorem 4.5]{Lad} for upper triangular DGAs. In Theorem \ref{1st attempt}, which is the first major step towards our goal, we obtain a dervived equivalence between $D(\Lambda)$ and $D(\mathscr{E})$, where $\mathscr{E}$ is the endomorphism DGA of a K-projective resoultion of the DG-module $T=\Sigma i_*X\oplus j_*j^*\Lambda$. We then follow this up by constructing a K-projective resoultion for $T$ in proposition \ref{Structure of P} which in turn is followed by the structure of the endomorphism DGA $\mathscr{E}$ in Proposition \ref{structure of E}. The remainder of this section is involved in the details of computing an quasi-isomorphisms between the DGA $\mathscr{E}$ and the upper triangular matrix DGA $\tilde{\Lambda}=\begin{bmatrix}S&\Hom_R(V,U)\\0&\Hom_R(U,U)^{\op}\end{bmatrix}$ with the final result being Theorem \ref{Main}, the main result of the paper which gives us a derived equivalnence between the upper triangular matrix DGAs $\Lambda$ and $\tilde{\Lambda}$. 

We start however with a statement of Keller's Theorem which we use in the proof of Theorem \ref{1st attempt}.

\begin{Theorem}[Keller's Theorem]
\label{Keller's Theorem}
Let $A$ be a DGA and let $N$ be a K-projective DG $A$-module which is compact in $D(A)$ such that $\langle N\rangle=D(A)$ and let $\mathscr{H}=\End_A(N,N)$. Then $D(A)\simeq D(\mathscr{H}^{\op})$.
\end{Theorem}

\begin{proof}
See \cite{Kel}.
\end{proof}

We are now able to make our first attempt at generalising \cite[Theorem 4.5]{Lad} for DGAs. For this we follow a similar method by introducing a DG-$\Lambda$-module $T=\Sigma i_*X\oplus j_*j^*\Lambda$.

\begin{Theorem}
\label{1st attempt}
Let $X$ be a DG $R$-module such that $_RX$ is compact and $\left\langle _RX\right\rangle=D(R)$. Let $\leftidx{_R}{M}{_S}$ be compact as a DG-$R$-module. Let $\mathscr{E}=\End_\Lambda(P)$, where $P$ is a K-projective resolution of $T=\Sigma i_*X\oplus j_*j^*\Lambda$. Then $\mathscr{E}$ is an DGA with $D(\Lambda)\simeq D(\mathscr{E}^{\op})$.
\end{Theorem}

\begin{proof}
Our aim to to apply Kellers theorem. To do this we need to show that $T$ is compact and that $\left\langle T \right\rangle\cong D(\Lambda)$. We begin with the compactness of $T$.

Since $T$ is a direct sum it is sufficent to show that both its direct summands $i_*X$ and $j_*j^*\Lambda$ are compact.

To show that $i_*X$ is compact we first note that by adjointness $$\Hom_{D(\Lambda)}(i_*X,\coprod A_k)\simeq\Hom_{D(R)}(X,i^!(\coprod A_k))$$ and that $i^!(\coprod A_k)=\RHom_\Lambda(B,\coprod A_k)$.
Also, since $B$ is compact it is not hard to show that $i^!(\coprod A_k)\cong\coprod i^!(A_k)$ in $D(\Lambda)$.

We therefore have that
$$\Hom_{D(\Lambda)}(i_*X,\coprod A_k)\simeq\Hom_{D(R)}(X,i^!(\coprod A_k))$$
$$\cong\Hom_{D(R)}(X,\coprod i^! A_k)\cong\coprod\Hom_{D(R)}(X,i^! A_k)$$
$$\simeq \coprod\Hom_{D(\Lambda)}(i_*X,A_k).$$
So $i_*X$ is compact as required.

\medskip

To show that $j_*j^*\Lambda$ is compact we observe from Lemma \ref{useful facts} that $j_*j^*\Lambda\cong \dfrac{C}{\left[\begin{smallmatrix} M\\0\end{smallmatrix}\right]}$. We know that $C$ is compact and since there is a distinguished triangle $\begin{bmatrix} M\\0\end{bmatrix}\rightarrow C\rightarrow \dfrac{C}{\left[\begin{smallmatrix} M\\0\end{smallmatrix}\right]}$ in $D(\Lambda)$ it is sufficent to show that $\begin{bmatrix} M\\0\end{bmatrix}$ is compact. 

Since $\begin{bmatrix} M\\0\end{bmatrix}\cong i_*M=B\stackrel{L}{\otimes}_RM$ and both $B$ and $M$ are compact we have that 
$$\Hom_{D(\Lambda)}(B\stackrel{L}{\otimes}_R M,\coprod A_k)\cong H^0\RHom_\Lambda(B\stackrel{L}{\otimes}_R M,\coprod A_k)$$
$$\cong H^0\RHom_R(M,\RHom_\Lambda(B,\coprod A_k))\cong H^0\RHom_R(M,\coprod\RHom_\Lambda(B,A_k))$$
$$\cong \coprod H^0\RHom_R(M,\RHom_\Lambda(B,A_k))\cong \coprod H^0\RHom_\Lambda(B\stackrel{L}{\otimes}_R M,A_k)$$
$$\cong \coprod\Hom_{D(\Lambda)}(i_*M,A_k).$$
So $\begin{bmatrix} M\\0\end{bmatrix}$ is compact and since $C$ is also compact we have that $j_*j^*\Lambda\cong\dfrac{C}{\left[\begin{smallmatrix} M\\0\end{smallmatrix}\right]}$ is compact, and so $T=\Sigma i_*X\oplus j_*j^*\Lambda$ is compact.

\bigskip 

It remains to show that $\left\langle T \right\rangle=D(\Lambda)$. For this it is sufficent to show that $\Lambda\in\left\langle T \right\rangle$.

Since $\Lambda\cong B\oplus C$ we only have to show that both $B$ and $C$ are in $\left\langle T\right\rangle$. 

To show that $B$ is contained in $\left\langle T\right\rangle$ we first observe that the functor $i_*(-)$ respects the operations of taking distinguished triangles, set indexed coproducts, quotients and suspensions. This gives us that $i_*(\left\langle X\right\rangle)\subseteq\left\langle i_*(X)\right\rangle$ for all $X\in D(R)$. Hence
$$B=i_*R\in\EssIm(i_*)=i_*(D(R))=i_*\left\langle X\right\rangle\subseteq\left\langle i_*X\right\rangle\subseteq\left\langle T\right\rangle.$$

\medskip

To show that $C\in \langle T\rangle$ we first observe that $\dfrac{C}{\left[\begin{smallmatrix} M\\0\end{smallmatrix}\right]}\cong j_*j^*\Lambda\in\langle T\rangle$ so if we can show that $\begin{bmatrix} M \\0 \end{bmatrix}\in\langle T\rangle$ then $C$ is in $\langle T \rangle$.
To show this we first observe that $\langle X\rangle=D(R)$ so $_RM$ can be built from $X$. Since $i_*$ preserves the possible constructions, we can build $\begin{bmatrix} M\\0\end{bmatrix}=i_*M$ from $i_*X\in\langle T\rangle$.

Hence we have that both $B$ and $C\in \langle T\rangle$ and therefore that $\Lambda\in\langle T\rangle$ so $\langle T\rangle=D(\Lambda)$.

We are now in a position to apply Keller's Therorem to get that $D(\Lambda)\simeq D(\mathscr{E}^{op})$.
\end{proof}

\bigskip

Our aim now is to find a K-projective resolution of $T$ in the above theorem so that we can calculate $\mathscr{E}$. For this we first need the following lemmas.

\bigskip

\begin{Lemma}
\label{K-projective fact}
Let $U$ be a K-projective resolution of a DG-$R$-module $X$. Then $\begin{bmatrix} U\\0 \end{bmatrix}$ is a K-projective resolution of $\begin{bmatrix} X\\0 \end{bmatrix}$ over $\Lambda$.
\end{Lemma}

\begin{proof}
Let $J$ be an exact DG-$\Lambda$-module. Then $$\Hom_\Lambda\left(\begin{bmatrix} U\\0 \end{bmatrix},J\right)\cong\Hom_\Lambda(B\otimes_RU,J)\cong\Hom_R(U,\Hom_\Lambda(B,J)).$$
Since both $U$ and $B$ are K-projective we have that this is exact and hence $\begin{bmatrix} U\\0 \end{bmatrix}$ is K-projective.
\end{proof} 

\bigskip

\begin{Lemma}
\label{K-proj mapping cones}
Let $f:X\rightarrow Y$ be a morphism of K-projective DG-modules over some DGA $R$ and let $Z$ be the mapping cone of $f$. Then $Z$ is also K-projective.
\end{Lemma}

\bigskip

\begin{Remark}
From definition of $_RM_S$ we have a quasi-isomorphism $_RV_S \stackrel{f}{\rightarrow}\leftidx{_R}{M}{_S}$ where $V$ is K-projective over $R$.
Also for the DG-$R$-module $_RX$ we can choose a quasi-isomorphism $\leftidx{_R}{U}\stackrel{g}{\rightarrow}\leftidx{_R}{X}$ where $U$ is a K-projective resolution. 
\end{Remark}

\bigskip

We can now prove the following proposition about the structure of $P$, a K-projective resolution of $T$.

\bigskip

\begin{Proposition}
\label{Structure of P}
Let $T=\Sigma i_*X\oplus j_*j^*\Lambda$ as defined in Theorem \ref{1st attempt}. Then $T$ has K-projective resolution $P=\Sigma\begin{bmatrix} U\\0 \end{bmatrix}\oplus W$ over $\Lambda$ where $W$ is the mapping cone of $\begin{bmatrix} V\\0 \end{bmatrix}\stackrel{\left[\begin{smallmatrix} f\\0 \end{smallmatrix}\right]}{\longrightarrow}\begin{bmatrix} M\\S \end{bmatrix}$.
\end{Proposition}

\begin{proof}
By lemma \ref{K-projective fact} we have that $\begin{bmatrix} U\\0 \end{bmatrix}$ is a K-projective resolution of $i_*X=\begin{bmatrix} X\\0 \end{bmatrix}$ and that $\begin{bmatrix} V\\0 \end{bmatrix}$ is a K-projective resolution of $\begin{bmatrix} M\\0 \end{bmatrix}$ over $\Lambda$.

\medskip

We now wish to find a K-projective resolution of $j_*j^*\Lambda$. To do this we first recall that $j_*j^*\Lambda=\dfrac{C}{\left[\begin{smallmatrix}M\\0\end{smallmatrix}\right]}$.

We now consider the map 
$$\begin{bmatrix} V\\0 \end{bmatrix}\stackrel{\left[\begin{smallmatrix} f\\0 \end{smallmatrix}\right]}{\longrightarrow}\begin{bmatrix} M\\S \end{bmatrix}$$ 
of DG-$\Lambda$-modules. This embeds into the distinguished triangle 
$$\begin{bmatrix} V\\0 \end{bmatrix}\stackrel{\left[\begin{smallmatrix} f\\0 \end{smallmatrix}\right]}{\longrightarrow}\begin{bmatrix} M\\S \end{bmatrix}\longrightarrow W.$$ 

We can now use this to obtain the diagram
$$\xymatrix{{\begin{bmatrix} M\\0 \end{bmatrix}} \ar[r] & C \ar[r] & {C/\left[\begin{smallmatrix} M\\0 \end{smallmatrix}\right]} \ar[r] & {}\\
{\begin{bmatrix} V\\0 \end{bmatrix}} \ar[u]^{\simeq} \ar[r] & C \ar@{=}[u] \ar[r] & W \ar@{.>}[u]^{\exists} \ar[r] & {}}$$
of distinguished triangles in $D(\Lambda)$ so there exists a quasi-isomorphism $W\rightarrow {C}/{\left[\begin{smallmatrix} M\\0 \end{smallmatrix}\right]}$.

By Lemma \ref{K-proj mapping cones} we also have that $W$ is K-projective and hence a K-projective resolution of $C/\left[\begin{smallmatrix} M\\0\end{smallmatrix}\right]\cong j_*j^*\Lambda$.

We now have K-projective resolutions for both direct summands of $T$ and hence $T$ has the K-projective resolution, $P=\Sigma\begin{bmatrix} U\\0\end{bmatrix}\oplus W$.

\end{proof}

\bigskip

Now that we have a K-projective resolution for $T$ in Theorem \ref{1st attempt} we can try to calculate the endomorphism DGA $\mathscr{E}=\End_\Lambda(T)$, but before we do so we need a few facts about $W$ as defined in the above propostion.

\bigskip

\begin{Remark}
$W$ is the mapping cone of $\begin{bmatrix}V\\0\end{bmatrix}\stackrel{\left[\begin{smallmatrix} f\\0 \end{smallmatrix}\right]}{\longrightarrow}\begin{bmatrix} M\\S \end{bmatrix}$ so $$W^\natural=\begin{bmatrix}M\\S\end{bmatrix}^\natural\oplus\Sigma\begin{bmatrix}V\\0\end{bmatrix}^\natural$$ i.e any element $w\in W^\natural$ is of the form $w=\begin{bmatrix}\left[\begin{smallmatrix}m\\s\end{smallmatrix}\right]\\\left[\begin{smallmatrix}v\\0\end{smallmatrix}\right]\end{bmatrix}$. In addition $W$ is equipped with the diffential
$$\partial^W=\begin{bmatrix}\partial^C&\left[\begin{smallmatrix}f&0\\0&0\end{smallmatrix}\right]\\[0.3cm]0&-\partial^{\left[\begin{smallmatrix}V\\0\end{smallmatrix}\right]}\end{bmatrix}.$$
and the quasi-ismorphism $W \rightarrow {C}/{\left[\begin{smallmatrix} M\\0 \end{smallmatrix}\right]}$ in the proof of proposition \ref{Structure of P} is give by $$\begin{bmatrix}\left[\begin{smallmatrix}m\\s\end{smallmatrix}\right]\\\left[\begin{smallmatrix}v\\0\end{smallmatrix}\right]\end{bmatrix}\mapsto\overline{\begin{bmatrix}0\\s\end{bmatrix}}.$$
\end{Remark}

\bigskip

\begin{Lemma}
\label{W quasi}
$W$ is isomorphic in $K(\Lambda)$ to $\begin{bmatrix}Z\\S\end{bmatrix}$ where $Z$ is the mapping cone of $f$. Furthermore $Z$ is exact.
\end{Lemma} 

\begin{proof}
Since $Z$ is the mapping cone of $f$ we have a distinguished triangle of the form:
$$\xymatrix{V \ar[r]^f & M\ar[r]^h & Z \ar[r]^k & \Sigma V \ar[r] &{}}.$$

We can use this to construct a distinguished triangle 
$$\xymatrix{{\begin{bmatrix}V\\0\end{bmatrix}} \ar[r]^{\left[\begin{smallmatrix}f\\0\end{smallmatrix}\right]} & {\begin{bmatrix}M\\S\end{bmatrix}}\ar[r]^{\left[\begin{smallmatrix}h\\1\end{smallmatrix}\right]} & {\begin{bmatrix}Z\\S\end{bmatrix}} \ar[r]^{\left[\begin{smallmatrix}k\\0\end{smallmatrix}\right]} & \Sigma{\begin{bmatrix}V\\0\end{bmatrix}} \ar[r] &{}}.$$
	
Which in turn we can use to obtain the diagram of distinguished triangles:
$$\xymatrix{{\begin{bmatrix}V\\0\end{bmatrix}} \ar[r]^{\left[\begin{smallmatrix}f\\0\end{smallmatrix}\right]} \ar@{=}[d] & {\begin{bmatrix}M\\S\end{bmatrix}}\ar[r]^{\left[\begin{smallmatrix}h\\1\end{smallmatrix}\right]} \ar@{=}[d]& {\begin{bmatrix}Z\\S\end{bmatrix}} \ar[r]^{\left[\begin{smallmatrix}k\\0\end{smallmatrix}\right]} \ar@{-->}[d]^{\exists} & \Sigma {\begin{bmatrix}V\\0\end{bmatrix}} \ar[r] \ar@{=}[d]& {}
\\{\begin{bmatrix}V\\0\end{bmatrix}} \ar[r]^{\left[\begin{smallmatrix}f\\0\end{smallmatrix}\right]} & {\begin{bmatrix}M\\S\end{bmatrix}} \ar[r] & W \ar[r] & \Sigma{\begin{bmatrix}V\\0\end{bmatrix}} \ar[r] &  {.}} $$

Hence we have that there exists an isomorphism $\begin{bmatrix}Z\\S\end{bmatrix}\rightarrow W$.

Finally since $Z$ is the mapping cone of a quasi-isomorphism it is exact.
\end{proof}

\bigskip

\begin{Lemma}
\label{Hom equiv}
Let $A$ and $B$ be DG-$R$-modules. Then $$\Hom_R(A,B)\cong\Hom_\Lambda\!\left(\begin{bmatrix}A\\0\end{bmatrix},\begin{bmatrix}B\\0\end{bmatrix}\right)$$ as complexes of abelian groups.

Furthermore $$\Hom_R(A,A)\cong\Hom_\Lambda\!\left(\begin{bmatrix}A\\0\end{bmatrix},\begin{bmatrix}A\\0\end{bmatrix}\right)$$ as DGA's.
\end{Lemma}

\begin{proof}
Define $\Theta:\Hom_R(A,B)\rightarrow\Hom_\Lambda\!\left(\begin{bmatrix}A\\0\end{bmatrix},\begin{bmatrix}B\\0\end{bmatrix}\right)$ by $\Theta(\phi)=\begin{bmatrix}\phi&0\\0&0\end{bmatrix}$. It is easy to see that $\Theta$ is an isomorphism of complexes of abelian groups. In addition in the case $B=A$, $\Theta$ becomes an isomorpism of DGA's. 
\end{proof}

\bigskip

The following proposition give the structure of $\mathscr{E}$ which by Theorem \ref{1st attempt} is dervived equivalent to the upper triangular matix DGA $\Lambda$.

\begin{Proposition}
\label{structure of E}
In the setup of Theorem \ref{1st attempt},
$$\mathscr{E}\cong\begin{bmatrix} \Hom_R(U,U) &  \Hom_\Lambda(W,\Sigma \left[\begin{smallmatrix} U \\0\end{smallmatrix}\right])\\ \Hom_\Lambda(\Sigma \left[\begin{smallmatrix} U \\0\end{smallmatrix}\right],W) & \Hom_\Lambda(W,W)\end{bmatrix}.$$
\end{Proposition}

\begin{proof}
Since $P=\Sigma U\oplus W$ consists of a direct sum we have that 
$$\mathscr{E}=\Hom_\Lambda(P,P)=\begin{bmatrix} \Hom_\Lambda(\Sigma \left[\begin{smallmatrix} U \\0\end{smallmatrix}\right],\Sigma \left[\begin{smallmatrix} U \\0\end{smallmatrix}\right]) &  \Hom_\Lambda(W,\Sigma \left[\begin{smallmatrix} U \\0\end{smallmatrix}\right])\\ \Hom_\Lambda(\Sigma \left[\begin{smallmatrix} U \\0\end{smallmatrix}\right],W) & \Hom_\Lambda(W,W)\end{bmatrix}.$$

\medskip

Furthermore from Lemma \ref{Hom equiv} above we have that $$\Hom_\Lambda\!\left(\Sigma\begin{bmatrix}U\\0\end{bmatrix},\Sigma\begin{bmatrix}U\\0\end{bmatrix}\right)\cong\Hom_\Lambda\!\left(\begin{bmatrix}U\\0\end{bmatrix},\begin{bmatrix}U\\0\end{bmatrix}\right)\cong\Hom_R(U,U).$$
\end{proof}

\bigskip

Our attention now is with obtaining a quasi-isomorphism between the entries of $\mathscr{E}$ and the corresponding entries $\tilde{\Lambda}=\begin{bmatrix}S&\Hom_R(V,U)\\0&\Hom_R(U,U)^{\op}\end{bmatrix}$ which will allow us to construct a isomorpism between the two DGAs.

\begin{Lemma}
\label{Exactness}
The complex of abelian groups $\Hom_\Lambda\!\left(\Sigma\begin{bmatrix}U\\0\end{bmatrix},W\right)$ is exact.
\end{Lemma}

\begin{proof}
For all $i$ we have that 
$$\H^i\Hom_\Lambda\left(\Sigma\begin{bmatrix}U\\0\end{bmatrix},W\right)\cong\H^0\Hom_\Lambda\left(\begin{bmatrix}U\\0\end{bmatrix},\Sigma^{i-1}W\right)$$
$$\cong\Hom_{K(\Lambda)}\left(\begin{bmatrix}U\\0\end{bmatrix},\Sigma^{i-1}W\right)\cong\Hom_{K(\Lambda)}\left(\begin{bmatrix}U\\0\end{bmatrix},\Sigma^{i-1}\begin{bmatrix}Z\\S\end{bmatrix}\right)$$

However for $\theta\in\Hom_\Lambda\left(\begin{bmatrix}U\\0\end{bmatrix},\Sigma^{i-1}\begin{bmatrix}Z\\S\end{bmatrix}\right)$ such that $\theta\left(\begin{bmatrix}u\\0\end{bmatrix}\right)=\begin{bmatrix}z\\s\end{bmatrix}$ for some $u\in U,z\in Z$ and $s\in S$, we have $$\begin{bmatrix}z\\s\end{bmatrix}
=\theta\left(\begin{bmatrix} u \\0\end{bmatrix}\right)
=\theta\left(\begin{bmatrix}1&0\\0&0\end{bmatrix}\begin{bmatrix}u\\0\end{bmatrix}\right)
=\begin{bmatrix}1&0\\0&0\end{bmatrix}\theta\left(\begin{bmatrix}u\\0\end{bmatrix}\right)$$
$$=\begin{bmatrix}1&0\\0&0\end{bmatrix}\begin{bmatrix}z\\s\end{bmatrix}
=\begin{bmatrix}z\\0\end{bmatrix}.$$ 

\smallskip

So $s=0$ and so $\theta\left(\begin{bmatrix}u\\0\end{bmatrix}\right)=\begin{bmatrix}z\\0\end{bmatrix}$. 

Hence $\Hom_\Lambda\!\left(\Sigma\begin{bmatrix}U\\0\end{bmatrix},\Sigma^{i-1}\begin{bmatrix}Z\\S\end{bmatrix}\right)\cong\Hom_\Lambda\!\left(\Sigma\begin{bmatrix}U\\0\end{bmatrix},\Sigma^{i-1}\begin{bmatrix}Z\\0\end{bmatrix}\right)$ and by Lemma \ref{Hom equiv} this is isomorphic to $\Hom_R(\Sigma U,\Sigma^{i-1}Z)$.

\smallskip

Taking this together with $U$ being K-projective and the exactness of $Z$ gives us that 
$$\Hom_{K(\Lambda)}\left(\begin{bmatrix}U\\0\end{bmatrix},\Sigma^{i-1}\begin{bmatrix}Z\\S\end{bmatrix}\right)\cong\Hom_{K(\Lambda)}\left(U,\Sigma^{i-1}Z\right)$$
$$\cong\Hom_{D(\Lambda)}\left(U,\Sigma^{i-1}Z\right)\cong 0,$$

\smallskip

Hence $\H^i\Hom_\Lambda\!\left(\Sigma\begin{bmatrix}U\\0\end{bmatrix},W\right)\cong 0$ for all $i$ and so $\Hom_\Lambda\!\left(\Sigma\begin{bmatrix}U\\0\end{bmatrix},W\right)$ is exact.
\end{proof}

\bigskip

\begin{Proposition}
\label{quasi S^op}
There is a quasi-isomorphism of DGAs $$\alpha:S^{\op}\rightarrow\Hom_\Lambda(W,W).$$
\end{Proposition}

\begin{proof} 
Define $\alpha:S^{\op}\rightarrow\Hom_\Lambda(W,W)$ by $\alpha(\tilde{s})=\begin{bmatrix} g_{\tilde{s}}&0\\0&l_{\tilde{s}}\end{bmatrix}$ where $g_{\tilde{s}}\left(\begin{bmatrix} m\\s\end{bmatrix}\right)=(-1)^{|\tilde{s}||s|}\begin{bmatrix} m\tilde{s}\\s\tilde{s}\end{bmatrix}$ and $l_{\tilde{s}}\left(\begin{bmatrix} v\\0\end{bmatrix}\right)=(-1)^{|\tilde{s}|(|v|+1)}\begin{bmatrix} v\tilde{s}\\0\end{bmatrix}$. 

We first want to show that $\alpha$ is a homomorphism of Differential Graded Algebras.

It is straightforward to check that $\alpha$ respects the operations of addition and multiplication. So it remains to check that $\alpha$ is compatible with the differential.

So $$\partial^{\Hom_\Lambda(W,W)}(\alpha(\tilde{s}))
=\partial^{\Hom_\Lambda(W,W)}\left(\begin{bmatrix}g_{\tilde{s}}&0\\0&l_{\tilde{s}}\end{bmatrix}\right)$$
$$=\partial^W\begin{bmatrix}g_{\tilde{s}}&0\\0&l_{\tilde{s}}\end{bmatrix}-(-1)^{|\tilde{s}|}\begin{bmatrix}g_{\tilde{s}}&0\\0&l_{\tilde{s}}\end{bmatrix}\partial^W$$
$$=\begin{bmatrix}\partial^C&\left[\begin{smallmatrix}f&0\\0&0\end{smallmatrix}\right]\\[0.3cm]0&-\partial^{\left[\begin{smallmatrix}V\\0\end{smallmatrix}\right]}\end{bmatrix}\begin{bmatrix}g_{\tilde{s}}&0\\0&l_{\tilde{s}}\end{bmatrix}-(-1)^{|\tilde{s}|}\begin{bmatrix}g_{\tilde{s}}&0\\0&l_{\tilde{s}}\end{bmatrix}\begin{bmatrix}\partial^C&\left[\begin{smallmatrix}f&0\\0&0\end{smallmatrix}\right]\\[0.3cm]0&-\partial^{\left[\begin{smallmatrix}V\\0\end{smallmatrix}\right]}\end{bmatrix}$$
$$=\begin{bmatrix}\partial^Cg_{\tilde{s}}&\left[\begin{smallmatrix}f&0\\0&0\end{smallmatrix}\right]l_{\tilde{s}}\\[0.3cm]0&-\partial^{\left[\begin{smallmatrix}V\\0\end{smallmatrix}\right]}l_{\tilde{s}}\end{bmatrix}-(-1)^{|\tilde{s}|}\begin{bmatrix}g_{\tilde{s}}\partial^C&g_{\tilde{s}}\left[\begin{smallmatrix}f&0\\0&0\end{smallmatrix}\right]\\[0.3cm]0&-l_{\tilde{s}}\partial^{\left[\begin{smallmatrix}V\\0\end{smallmatrix}\right]}\end{bmatrix}$$
$$=\begin{bmatrix}\partial^Cg_{\tilde{s}}-(-1)^{|\tilde{s}|}g_{\tilde{s}}\partial^C&\left[\begin{smallmatrix}f&0\\0&0\end{smallmatrix}\right]l_{\tilde{s}}-(-1)^{|\tilde{s}|}g_{\tilde{s}}\left[\begin{smallmatrix}f&0\\0&0\end{smallmatrix}\right]\\[0.3cm]0&-\partial^{\left[\begin{smallmatrix}V\\0\end{smallmatrix}\right]}l_{\tilde{s}}+(-1)^{|\tilde{s}|}l_{\tilde{s}}\partial^{\left[\begin{smallmatrix}V\\0\end{smallmatrix}\right]}\end{bmatrix}.$$

\smallskip

Considering each of the terms in this matrix seperately, starting with the upper left term, we get:
$$\left(\partial^Cg_{\tilde{s}}-(-1)^{|\tilde{s}|}g_{\tilde{s}}\partial^C\right)\left(\begin{bmatrix}m\\s\end{bmatrix}\right)$$
$$=\partial^C\left((-1)^{|\tilde{s}||s|}\begin{bmatrix}m\tilde{s}\\s\tilde{s}\end{bmatrix}\right)-(-1)^{|\tilde{s}|}g_{\tilde{s}}\left(\begin{bmatrix}\partial^M(m)\\\partial^S(s)\end{bmatrix}\right)$$
$$=(-1)^{|\tilde{s}||s|}\begin{bmatrix}\partial^M(m\tilde{s})\\\partial^S(s\tilde{s})\end{bmatrix}-(-1)^{|\tilde{s}|}(-1)^{|\tilde{s}|(|s|-1)}\begin{bmatrix}\partial^M(m)\tilde{s}\\\partial^S(s)\tilde{s}\end{bmatrix}$$
$$=(-1)^{|\tilde{s}||s|}\begin{bmatrix}\partial^M(m)\tilde{s}+(-1)^{|s|}m\partial^S(\tilde{s})\\\partial^S(s)\tilde{s}+(-1)^{|s|}s\partial^S(\tilde{s})\end{bmatrix}-(-1)^{|\tilde{s}||s|}\begin{bmatrix}\partial^M(m)\tilde{s}\\\partial^S(s)\tilde{s}\end{bmatrix}$$
$$=(-1)^{|\tilde{s}||s|}\begin{bmatrix}(-1)^{|s|}m\partial^S(\tilde{s})\\(-1)^{|s|}s\partial^S(\tilde{s})\end{bmatrix}=(-1)^{|\partial^S(\tilde{s})||s|}\begin{bmatrix}m\partial^S(\tilde{s})\\s\partial^S(\tilde{s})\end{bmatrix}$$
$$=g_{\partial^S(\tilde{s})}\left(\begin{bmatrix}m\\s\end{bmatrix}\right).$$

So $\partial^Cg_{\tilde{s}}-(-1)^{|\tilde{s}|}g_{\tilde{s}}\partial^C=g_{\partial^S(\tilde{s})}$.

\smallskip

By a similar arguement we also have for the lower right entry that $-\partial^{\left[\begin{smallmatrix}V\\0\end{smallmatrix}\right]}l_{\tilde{s}}+(-1)^{|\tilde{s}|}l_{\tilde{s}}\partial^{\left[\begin{smallmatrix}V\\0\end{smallmatrix}\right]}=l_{\partial^S(\tilde{s})}$.

\smallskip

Finally for the upper right entry we have:
$$\left(\begin{bmatrix}f&0\\0&0\end{bmatrix}l_{\tilde{s}}-(-1)^{|\tilde{s}|}g_{\tilde{s}}\begin{bmatrix}f&0\\0&0\end{bmatrix}\right)\left(\begin{bmatrix}v\\0\end{bmatrix}\right)$$
$$=\begin{bmatrix}f&0\\0&0\end{bmatrix}\left((-1)^{|\tilde{s}|(|v|+1)}\begin{bmatrix}v\tilde{s}\\0\end{bmatrix}\right)-(-1)^{|\tilde{s}|}g_{\tilde{s}}\left(\begin{bmatrix}f(v)\\0\end{bmatrix}\right)$$
$$(-1)^{|\tilde{s}|(|v|+1)}\begin{bmatrix}f(v\tilde{s})\\0\end{bmatrix}-(-1)^{|\tilde{s}|}(-1)^{|\tilde{s}||v|}\begin{bmatrix}f(v)\tilde{s}\\0\end{bmatrix}$$
$$(-1)^{|\tilde{s}|(|v|+1)}\begin{bmatrix}f(v)\tilde{s}\\0\end{bmatrix}-(-1)^{|\tilde{s}|(|v|+1)}\begin{bmatrix}f(v)\tilde{s}\\0\end{bmatrix}=0.$$

Substituting these values back into the matrix gives us that:
$$\partial^{\Hom_\Lambda(W,W)}(\alpha(\tilde{s}))=\begin{bmatrix}g_{\partial^S(\tilde{s})}&0\\0&l_{\partial^S(\tilde{s})}\end{bmatrix}=\alpha(\partial^S(\tilde{s})).$$

Hence we have that $\alpha$ is a homomorphism of Differential Graded Algebras. It remains to show that it is also a quasi-isomorphism.

\medskip

Now let $\theta\in\Hom_\Lambda\!\left(W,C/\left[\begin{smallmatrix}M\\0\end{smallmatrix}\right]\right)$ such that $\theta\left(\begin{bmatrix}\left[\begin{smallmatrix}m\\s\end{smallmatrix}\right]\\\left[\begin{smallmatrix}v\\0\end{smallmatrix}\right]\end{bmatrix}\right)=\overline{\begin{bmatrix}0\\s\end{bmatrix}}$. This quasi-isomorphism gives us the homomorphism of complexes of abelian groups $$\Hom_\Lambda(W,\theta):\Hom_\Lambda(W,W)\rightarrow\Hom_\Lambda\left(W,C/\left[\begin{smallmatrix}M\\0\end{smallmatrix}\right]\right).$$

Define $\beta=\Hom_\Lambda(W,\theta)\circ\alpha:S^{\op}\rightarrow\Hom_\Lambda\left(W,C/\left[\begin{smallmatrix}M\\0\end{smallmatrix}\right]\right)$. 
Then $\beta$ is a homomorphism of complexes of abelian groups with
$$\beta(\tilde{s})\left(\begin{bmatrix}\left[\begin{smallmatrix}m\\s\end{smallmatrix}\right]\\\left[\begin{smallmatrix}v\\0\end{smallmatrix}\right]\end{bmatrix}\right)
=\theta\left((-1)^{|s||\tilde{s}|}\begin{bmatrix}\left[\begin{smallmatrix}m\tilde{s}\\s\tilde{s}\end{smallmatrix}\right]\\\left[\begin{smallmatrix}v\\0\end{smallmatrix}\right]\end{bmatrix}\right)
=(-1)^{|s||\tilde{s}|}\overline{\begin{bmatrix}0\\s\tilde{s}\end{bmatrix}}.$$

Now let $\psi\in \Hom_{\Lambda}\left(W,C/\left[\begin{smallmatrix}M\\0\end{smallmatrix}\right]\right)$. Then for any $m\in M, s\in S$ and $v\in V$ we have that $\psi\left(\begin{bmatrix}\left[\begin{smallmatrix}m\\s\end{smallmatrix}\right]\\\left[\begin{smallmatrix}v\\0\end{smallmatrix}\right]\end{bmatrix}\right)
=\overline{\begin{bmatrix}0\\\tilde{s}\end{bmatrix}}$ for some $\tilde{s}\in S$. Hence we have $$\psi\left(\begin{bmatrix}\left[\begin{smallmatrix}m\\s\end{smallmatrix}\right]\\\left[\begin{smallmatrix}v\\0\end{smallmatrix}\right]\end{bmatrix}\right)
=\overline{\begin{bmatrix}0\\\tilde{s}\end{bmatrix}}
=\begin{bmatrix}0&0\\0&1\end{bmatrix}\overline{\begin{bmatrix}0\\\tilde{s}\end{bmatrix}}
=\begin{bmatrix}0&0\\0&1\end{bmatrix}\psi\left(\begin{bmatrix}\left[\begin{smallmatrix}m\\s\end{smallmatrix}\right]\\\left[\begin{smallmatrix}v\\0\end{smallmatrix}\right]\end{bmatrix}\right)=\psi\left(\begin{bmatrix}\left[\begin{smallmatrix}0\\s\end{smallmatrix}\right]\\\left[\begin{smallmatrix}0\\0\end{smallmatrix}\right]\end{bmatrix}\right).$$ 

Furthermore  
$$\psi\left(\begin{bmatrix}\left[\begin{smallmatrix}m\\s\end{smallmatrix}\right]\\\left[\begin{smallmatrix}v\\0\end{smallmatrix}\right]\end{bmatrix}\right)
=\psi\left(\begin{bmatrix}\left[\begin{smallmatrix}0\\s\end{smallmatrix}\right]\\\left[\begin{smallmatrix}0\\0\end{smallmatrix}\right]\end{bmatrix}\right)
=\begin{bmatrix}0&0\\0&s\end{bmatrix}\psi\left(\begin{bmatrix}\left[\begin{smallmatrix}0\\1\end{smallmatrix}\right]\\\left[\begin{smallmatrix}0\\0\end{smallmatrix}\right]\end{bmatrix}\right).$$ 

So each element of $\Hom_{\Lambda}\left(W,C/\left[\begin{smallmatrix}M\\0\end{smallmatrix}\right]\right)$ depends entirely on where it sends $\begin{bmatrix}\left[\begin{smallmatrix}0\\1\end{smallmatrix}\right]\\\left[\begin{smallmatrix}0\\0\end{smallmatrix}\right]\end{bmatrix}$.

\medskip

We therefore have that for every $s\in S$ we have that $\beta(s)$ is the element of $\Hom_{\Lambda}\left(W,C/\left[\begin{smallmatrix}M\\0\end{smallmatrix}\right]\right)$ which sends $\begin{bmatrix}\left[\begin{smallmatrix}0\\1\end{smallmatrix}\right]\\\left[\begin{smallmatrix}0\\0\end{smallmatrix}\right]\end{bmatrix}$ 
to $\overline{\begin{bmatrix}0\\s\end{bmatrix}}$.

Since elements of $\Hom_{\Lambda}\left(W,C/\left[\begin{smallmatrix}M\\0\end{smallmatrix}\right]\right)$ depend entirely on where they send $\begin{bmatrix}\left[\begin{smallmatrix}0\\1\end{smallmatrix}\right]\\\left[\begin{smallmatrix}0\\0\end{smallmatrix}\right]\end{bmatrix}$ 
and since $\overline{\left[\begin{matrix}0\\s\end{matrix}\right]}\neq\overline{\left[\begin{matrix}0\\s'\end{matrix}\right]}$ for all $s,s'\in S$ with $s\neq s'$ we have that $\beta$ is a bijection and so a isomorphism of complexes of abelian groups.

Furthermore since $W$ is K-projective and $\theta$ is a quasi-isomorphism we have that $\Hom_\Lambda(W,\theta)$ is a quasi-isomorphism and therefore since $\beta$ is an isomorphism we have that $\alpha$ must also be a quasi-isomorphism. 
\end{proof}

\bigskip

\begin{Lemma}
\label{Hom iso2}
There exists a quasi-isomorphism $\Psi:\Hom_R(V,U)\rightarrow\Hom_\Lambda\left(W,\Sigma\begin{bmatrix}U\\0\end{bmatrix}\right)$ of complexes of abelian groups, such that $$\Psi(\theta)\left(\begin{bmatrix}\left[\begin{smallmatrix}m\\s\end{smallmatrix}\right]\\\left[\begin{smallmatrix}v\\0\end{smallmatrix}\right]\end{bmatrix}\right)=(-1)^{|\theta|}\begin{bmatrix}\theta(v)\\0\end{bmatrix}.$$
\end{Lemma}

\begin{proof}
Consider the distinguished triangle $$\begin{bmatrix}V\\0\end{bmatrix}\stackrel{\left[\begin{smallmatrix}f\\0\end{smallmatrix}\right]}{\rightarrow}\begin{bmatrix}M\\S\end{bmatrix}\stackrel{\iota}{\rightarrow}W\stackrel{\pi}{\rightarrow}\Sigma\begin{bmatrix}V\\0\end{bmatrix}\rightarrow$$
in $K(\Lambda)$.

Since $W$ is the mapping cone of $\begin{bmatrix}f\\0\end{bmatrix}$ we have that $\pi$ is given by $\pi\left(\begin{bmatrix}\left[\begin{smallmatrix}m\\s\end{smallmatrix}\right]\\\left[\begin{smallmatrix}v\\0\end{smallmatrix}\right]\end{bmatrix}\right)=\begin{bmatrix}v\\0\end{bmatrix}$.

By applying the functor $\Hom_\Lambda\left(-,\Sigma\begin{bmatrix}U\\0\end{bmatrix}\right)$ we get a distinguished triangle

$$\leftarrow\Hom_\Lambda\left(\begin{bmatrix}M\\S\end{bmatrix},\Sigma\begin{bmatrix}U\\0\end{bmatrix}\right)
\leftarrow\Hom_\Lambda\left(W,\Sigma\begin{bmatrix}U\\0\end{bmatrix}\right)$$
$$\stackrel{\pi^*}{\leftarrow}\Hom_\Lambda\left(\Sigma\begin{bmatrix}V\\0\end{bmatrix},\Sigma\begin{bmatrix}U\\0\end{bmatrix}\right)
\leftarrow\Hom_\Lambda\left(\Sigma\begin{bmatrix}M\\S\end{bmatrix},\Sigma\begin{bmatrix}U\\0\end{bmatrix}\right)$$
in $K(\Ab)$.

Now let $\theta\in\Hom_\Lambda\left(\begin{bmatrix}M\\S\end{bmatrix},\Sigma^i\begin{bmatrix}U\\0\end{bmatrix}\right)$. Then, since $\begin{bmatrix}M\\S\end{bmatrix}$ is generated by $\begin{bmatrix}0\\1\end{bmatrix}$ as a DG-$\Lambda$-module, we have that $\theta$ depends entirely upon where it sends $\begin{bmatrix}0\\1\end{bmatrix}$.

Let $\theta\left(\begin{bmatrix}0\\1\end{bmatrix}\right)=\begin{bmatrix}u\\0\end{bmatrix}$. Then

$$\begin{bmatrix}u\\0\end{bmatrix}
=\theta\left(\begin{bmatrix}0\\1\end{bmatrix}\right)
=\theta\left(\begin{bmatrix}0&0\\0&1\end{bmatrix}\begin{bmatrix}0\\1\end{bmatrix}\right)
=\begin{bmatrix}0&0\\0&1\end{bmatrix}\theta\left(\begin{bmatrix}0\\1\end{bmatrix}\right)
=\begin{bmatrix}0&0\\0&1\end{bmatrix}\begin{bmatrix}u\\0\end{bmatrix}=0,$$

so $\theta=0$ and hence $\Hom_\Lambda\left(\begin{bmatrix}M\\S\end{bmatrix},\Sigma^i\begin{bmatrix}U\\0\end{bmatrix}\right)=0$ for all $i$.

Hence the distinguished triangle above shows that
$$\pi^*:\Hom_\Lambda\left(\Sigma\begin{bmatrix}V\\0\end{bmatrix},\Sigma\begin{bmatrix}U\\0\end{bmatrix}\right)\rightarrow\Hom_\Lambda\left(W,\Sigma\begin{bmatrix}U\\0\end{bmatrix}\right)$$
is a quasi-isomorphism.

We can now use this along with the suspension $\Sigma$ and the isomorphism $\Theta$ defined in the proof of Lemma \ref{Hom equiv} to obtain the diagram
$$\Hom_R(V,U)\stackrel{\Theta}{\longrightarrow}\Hom_{\Lambda}\left(\begin{bmatrix}V\\0\end{bmatrix},\begin{bmatrix}U\\0\end{bmatrix}\right)\stackrel{\Sigma(-)}{\longrightarrow}$$ $$\Hom_{\Lambda}\left(\Sigma\begin{bmatrix}V\\0\end{bmatrix},\Sigma\begin{bmatrix}U\\0\end{bmatrix}\right)\stackrel{\pi^*}{\longrightarrow}\Hom_{\Lambda}\left(W,\Sigma\begin{bmatrix}U\\0\end{bmatrix}\right)$$

Since each of the maps in the diagram is a quasi-isomorphism we can use them to define the quasi-isomorphism $\Psi:\Hom_R(V,U)\rightarrow\Hom_\Lambda\!\left(W,\Sigma\begin{bmatrix}U\\0\end{bmatrix}\right)$ by the composition $$\Psi=\pi^*\circ\Sigma(-)\circ\Theta.$$

Finally for $\theta\in\Hom_R(V,U)$ we have that
$$\Psi(\theta)=\pi^*\circ\Sigma\circ\Theta(\theta)$$
$$=\pi^*\Sigma\left(\begin{bmatrix}\theta&0\\0&0\end{bmatrix}\right)$$
$$=\pi^*\left((-1)^{|\theta|}\begin{bmatrix}\theta&0\\0&0\end{bmatrix}\right)$$
$$=(-1)^{|\theta|}\begin{bmatrix}\theta&0\\0&0\end{bmatrix}\circ\pi.$$

So for $\begin{bmatrix}\left[\begin{smallmatrix}m\\s\end{smallmatrix}\right]\\\left[\begin{smallmatrix}v\\0\end{smallmatrix}\right]\end{bmatrix}\in W$, we have that 
$$\Psi(\theta)\left(\begin{bmatrix}\left[\begin{smallmatrix}m\\s\end{smallmatrix}\right]\\\left[\begin{smallmatrix}v\\0\end{smallmatrix}\right]\end{bmatrix}\right)
=(-1)^{|\theta|}\begin{bmatrix}\theta&0\\0&0\end{bmatrix}\circ\pi\left(\begin{bmatrix}\left[\begin{smallmatrix}m\\s\end{smallmatrix}\right]\\\left[\begin{smallmatrix}v\\0\end{smallmatrix}\right]\end{bmatrix}\right)$$
$$=(-1)^{|\theta|}\begin{bmatrix}\theta&0\\0&0\end{bmatrix}\left(\begin{bmatrix}v\\0\end{bmatrix}\right)
=(-1)^{|\theta|}\begin{bmatrix}\theta(v)\\0\end{bmatrix}.$$

\end{proof}

\bigskip

\begin{Remark}
From the right DG-$S$-module structure on $V$ we have that $\Hom_R(V,U)$ is a left DG-$S$-module. In addition $\Hom_R(V,U)$ is a left DG-$\Hom_R(U,U)$-module.
Hence we have that $\begin{bmatrix}S&\Hom_R(V,U)\\0&\Hom_R(U,U)^{\op}\end{bmatrix}$ is a DGA. 
\end{Remark}

\bigskip

We are now in a position to produce our main Theorem, a version of \cite[Theorem 4.5]{Lad} for DGAs. 

\begin{Theorem}
\label{Main}
Let $X$ be a DG $R$-module such that $_RX$ is compact with $\left\langle _RX\right\rangle=D(R)$ and let $\leftidx{_R}{M}{_S}$ be compact as a DG-$R$-module. 
Then for the upper triangular differential graded algebras  
$$\Lambda=\begin{bmatrix}R&M\\0&S\end{bmatrix}\textrm{ and  } \tilde{\Lambda}=\begin{bmatrix}S&\Hom_R(V,U)\\0&\Hom_R(U,U)^{\op}\end{bmatrix}$$
we have that $D(\Lambda)\simeq D(\tilde{\Lambda})$.
\end{Theorem}

\begin{proof}
From Theorem \ref{1st attempt} and lemma \ref{structure of E} we have that $D(\Lambda)\simeq D(\mathscr{E}^{\op})$ where $$\mathscr{E}=\begin{bmatrix}\Hom_R(U,U)&\Hom_\Lambda(W,\Sigma\left[\begin{smallmatrix}U\\0\end{smallmatrix}\right])\\\Hom_R(\Sigma \left[\begin{smallmatrix}U\\0\end{smallmatrix}\right],W)&\Hom_\Lambda(W,W)\end{bmatrix}.$$ 
\smallskip

We therefore only need to show that there is a quasi-isomorphism of DGA's from 
$\tilde{\Lambda}^{\op}=\begin{bmatrix}\Hom_R(U,U)&\Hom_R(V,U)\\0&S^{\op}\end{bmatrix}$ to $\mathscr{E}$.

\medskip

From proposition \ref{quasi S^op} we have that there exists a quasi-isomorphism $\alpha:S^{\op}\rightarrow\Hom_\Lambda(W,W)$. 
Hence we can define the map 
$$\Phi:\begin{bmatrix}\Hom_R(U,U)&\Hom_R(V,U)\\0&S^{\op}\end{bmatrix}
\rightarrow
\begin{bmatrix}\Hom_R(U,U)&\Hom_\Lambda(W,\Sigma\left[\begin{smallmatrix}U\\0\end{smallmatrix}\right])\\\Hom_R(\Sigma \left[\begin{smallmatrix}U\\0\end{smallmatrix}\right],W)&\Hom_\Lambda(W,W)\end{bmatrix}$$
by $\Phi\left(\begin{bmatrix}\phi&\theta\\0&s\end{bmatrix}\right)=\left(\begin{bmatrix}\phi&(-1)^{|\theta|}\Psi(\theta)\\0&\alpha(s)\end{bmatrix}\right)$.

Here $\Psi:\Hom_R(V,U)\rightarrow\Hom_\Lambda\!\left(W,\Sigma\begin{bmatrix}U\\0\end{bmatrix}\right)$ is the quasi-isomorphism from Lemma \ref{Hom iso2}.

We now need to show that $\Phi$ is a morphism of DGA's.

\smallskip

Both addition and compatibility with the differential follow from the fact that $\alpha$ is a morphism of diffential graded algebras. So we only need to check multiplication:

Let . denote multiplication in $S^{\op}$. Let $\begin{bmatrix}\phi&\theta\\0&s\end{bmatrix}\in \tilde{\Lambda}_i$ and $\begin{bmatrix}\phi'&\theta'\\0&s'\end{bmatrix}\in \tilde{\Lambda}_j$; then we have
$$\Phi\left(\begin{bmatrix}\phi&\theta\\0&s\end{bmatrix}\right)\Phi\left(\begin{bmatrix}\phi'&\theta'\\0&s'\end{bmatrix}\right)
=\begin{bmatrix}\phi&(-1)^i\Psi(\theta)\\0&\alpha(s)\end{bmatrix}\begin{bmatrix}\phi'&(-1)^j\Psi(\theta')\\0&\alpha(s')\end{bmatrix}$$
$$=\begin{bmatrix}\phi\phi'&(-1)^j\phi\Psi(\theta')+(-1)^i\Psi(\theta)\alpha(s')\\0&\alpha(s)\alpha(s')\end{bmatrix}$$
$$=\begin{bmatrix}\phi\phi'&(-1)^j\phi\Psi(\theta')+(-1)^i\Psi(\theta)\alpha(s')\\0&\alpha(s.s')\end{bmatrix}$$
and
$$\Phi\left(\begin{bmatrix}\phi&\theta\\0&s\end{bmatrix}\begin{bmatrix}\phi'&\theta'\\0&s'\end{bmatrix}\right)
=\Phi\left(\begin{bmatrix}\phi\phi'&\phi\theta+\theta.s'\\0&s.s'\end{bmatrix}\right)$$
$$=\begin{bmatrix}\phi\phi'&(-1)^{(i+j)}\Psi(\phi\theta'+(-1)^{ij}s'\theta)\\0&\alpha(s.s')\end{bmatrix}.$$

However
$$(-1)^{(i+j)}\Psi((\phi\theta'+(-1)^{ij}s'\theta))\left(\begin{bmatrix}\left[\begin{smallmatrix}m\\s\end{smallmatrix}\right]\\\left[\begin{smallmatrix}v\\0\end{smallmatrix}\right]\end{bmatrix}\right)$$
$$=(-1)^{(i+j)}\Psi(\phi\theta')\left(\begin{bmatrix}\left[\begin{smallmatrix}m\\s\end{smallmatrix}\right]\\\left[\begin{smallmatrix}v\\0\end{smallmatrix}\right]\end{bmatrix}\right)+(-1)^{(i+j)}(-1)^{ij}\Psi(s'\theta)\left(\begin{bmatrix}\left[\begin{smallmatrix}m\\s\end{smallmatrix}\right]\\\left[\begin{smallmatrix}v\\0\end{smallmatrix}\right]\end{bmatrix}\right)$$
$$=(-1)^{(i+j)}(-1)^{(i+j)}\begin{bmatrix}\phi\theta'(v)\\0\end{bmatrix}+(-1)^{(i+j)}(-1)^{ij}(-1)^{(i+j)}\left(\begin{bmatrix}(s'\theta)(v)\\0\end{bmatrix}\right)$$
$$=\phi\begin{bmatrix}\theta'(v)\\0\end{bmatrix}+(-1)^{ij}(-1)^{j(i+(i+1))}\begin{bmatrix}\theta(vs')\\0\end{bmatrix}$$
$$=\phi\begin{bmatrix}\theta'(v)\\0\end{bmatrix}+(-1)^{j(i+1)}\begin{bmatrix}\theta(vs')\\0\end{bmatrix}$$
$$=(-1)^j\phi\Psi(\theta')\left(\begin{bmatrix}\left[\begin{smallmatrix}m\\s\end{smallmatrix}\right]\\\left[\begin{smallmatrix}v\\0\end{smallmatrix}\right]\end{bmatrix}\right)+(-1)^{j(i+1)}(-1)^i\Psi(\theta)\left(\begin{bmatrix}\left[\begin{smallmatrix}ms'\\ss'\end{smallmatrix}\right]\\\left[\begin{smallmatrix}vs'\\0\end{smallmatrix}\right]\end{bmatrix}\right)$$
$$=((-1)^j\phi\Psi(\theta')+(-1)^i\Psi(\theta)\alpha(s'))\left(\begin{bmatrix}\left[\begin{smallmatrix}m\\s\end{smallmatrix}\right]\\\left[\begin{smallmatrix}v\\0\end{smallmatrix}\right]\end{bmatrix}\right),$$

\smallskip

so $(-1)^j\phi\Psi(\theta')+(-1)^i\Psi(\theta)\alpha(s')=(-1)^{(i+j)}\Psi((\phi\theta'+(-1)^{ij}s'\theta))$ and therefore $$\Phi\left(\begin{bmatrix}\phi&\theta\\0&s\end{bmatrix}\right)\Phi\left(\begin{bmatrix}\phi'&\theta'\\0&s'\end{bmatrix}\right)=\Phi\left(\begin{bmatrix}\phi&\theta\\0&s\end{bmatrix}\begin{bmatrix}\phi'&\theta'\\0&s'\end{bmatrix}\right).$$

\medskip

We therefore have that $\Phi$ is a morphism of Diffential Graded Algebras. Furthermore since from lemma \ref{Exactness} we have that $\Hom_\Lambda\!\left(\Sigma\begin{bmatrix}U\\0\end{bmatrix},W\right)$ is exact and so the map $0\rightarrow\Hom_\Lambda\!\left(\Sigma\begin{bmatrix}U\\0\end{bmatrix},W\right)$ is a quasi isomorphism. 

Taking these together with the fact that $\alpha:S^{\op}\rightarrow\Hom_\Lambda(W,W)$ is a quasi-isomorphism we have that $\Phi$ is a quasi-isomorphism.

Hence $\mathscr{E}\simeq\tilde{\Lambda}^{\op}$ and so $$D(\Lambda)\simeq D(\mathscr{E}^{\op})\simeq D(\tilde{\Lambda}).$$

\end{proof}

\section{Examples}

We shall now conclude with some examples. In the first example we will show that by taking $R$ and $S$ to be $k$-algebras and making the same assumptions as in \cite{Lad}, we obtain what is in essence the same result.

\begin{Definition}
An $R$-module $X$ is called rigid if $\Ext_R^i(X,X)=0$ for all $i\neq 0$.
\end{Definition}

\medskip

\begin{Theorem}
Let $R$ and $S$ be rings and $\leftidx{_R}{M}{_S}$ a $R$-$S$-bimodule such that $\leftidx{_R}{M}$ is compact in $D(R)$ and when $R$ and $S$ are considered as DGAs then $\leftidx{_R}{M}{_S}$ as a DG-bimodule is quasi-isomorphic to $\leftidx{_R}{V}{_S}$ which is a K-projective DG-$R$-module. Let $\leftidx{_R}{X}$ be a compact and rigid $R$-module with $\langle X\rangle=D(R)$ and $\Ext_R^n(\leftidx{_R}{M},\leftidx{_R}{X})=0$ for all $n\neq 0$. Then the triangular matrix rings 
$$\Lambda=\begin{bmatrix}R&M\\0&S\end{bmatrix} \textrm{ and } \tilde{\Lambda}=\begin{bmatrix}S&\Hom_R(M,X)\\0&\End_R(X)^{\op}\end{bmatrix}$$
are derived equivalent.  
\end{Theorem}
\begin{proof}
By considering the rings $R$ and $S$ and modules $M$ and $X$ to be DGA's and DG-modules respectively we can apply Theorem \ref{Main} to get that the DGA's 
$$\begin{bmatrix}R&M\\0&S\end{bmatrix}\textrm{ and } \begin{bmatrix}S&\Hom_R(V,U)\\0&\Hom_R(U,U)^{\op}\end{bmatrix}$$
are derived equivalent, where $U$ is a K-projective resolution of $X$.

\medskip

Since $\Hom_R(U,U)=\RHom_R(X,X)$ we have that 
$$H^i\Hom_R(U,U)=H^i\RHom_R(X,X)=\Ext_R^i(X,X)=0$$
for all $i\neq 0$ since $X$ is rigid and 
$$H^0\Hom_R(U,U)=H^0\RHom_R(X,X)=\End_R(X).$$ 
Similarly since $\Hom_R(V,U)=\RHom_R(M,X)$ we have that
$$H^i\Hom_R(V,U)=H^i\RHom_R(M,X)=\Ext_R^i(M,X)=0$$
for all $i\neq 0$ and 
$$H^0\Hom_R(V,U)=H^0\RHom_R(M,X)=\Hom_R(M,X).$$

Hence we have that $H^i\begin{bmatrix}S&\Hom_R(V,U)\\0&\Hom_R(U,U)^{\op}\end{bmatrix}=0$ for all $i\neq 0$ and $$H^0\begin{bmatrix}S&\Hom_R(V,U)\\0&\Hom_R(U,U)^{\op}\end{bmatrix}=\begin{bmatrix}S&\Hom_R(M,X)\\0&\End_R(X)^{\op}\end{bmatrix}.$$

We therefore have that the matrix ring $\begin{bmatrix}S&\Hom_R(M,X)\\0&\End_R(X)^{\op}\end{bmatrix}$ is derived equivalent to the DGA $\begin{bmatrix}S&\Hom_R(V,U)\\0&\Hom_R(U,U)^{\op}\end{bmatrix}$ and so derived equivalent to the matrix ring $\begin{bmatrix}R&M\\0&S\end{bmatrix}$.

\end{proof} 

\bigskip

Our next example considers the special case obtained when we take $\leftidx{_R}{X}=\leftidx{_R}{R}$.

\begin{Corollary}
\label{X=R}
Let $\leftidx{_R}{M}{_S}$ be compact as a DG-$R$-module. Then the triangular matrix DGAs 
$$\Lambda=\begin{bmatrix}R&M\\0&S\end{bmatrix} \textrm{ and } \tilde{\Lambda}=\begin{bmatrix}S&\Hom_R(V,R)\\0&R\end{bmatrix}$$
where $V$ is K-projective over $R$ and is quasi-isomorphic to $\leftidx{_R}{M}{_S}$, are derived equivalent.
\end{Corollary}

\bigskip

For the next example we require the idea of the duality on $\Mod R$ which we define next.

\begin{Definition}
Let $R$ be a finite dimensional DGA over a field $k$. Then we can define the duality on $\Mod R$ by $D:\Mod R\rightarrow\Mod R^{\op}$ where $D(-)=\Hom_k(-,k)$.
\end{Definition}

\bigskip

The final example below considers the case where the DGA's $R$ and $S$ are over some field $k$ and $R$ is self dual in the sense of the above definition.

\begin{Theorem}
Let $R$ be a finite dimensional and self dual in the sense that $DR\cong R$ in the derived category of DG-bi-$R$-modules and let $\leftidx{_R}{M}{_S}$ be compact as a DG-$R$-module. Then 
$$\Lambda=\begin{bmatrix}R&M\\0&S\end{bmatrix} \textrm{ and } \tilde{\Lambda}=\begin{bmatrix}S&DM\\0&R\end{bmatrix}$$
are derived equivalent.
\end{Theorem}
\begin{proof}
From corollary \ref{X=R} we have that 
$$\begin{bmatrix}R&M\\0&S\end{bmatrix} \textrm{ and } \begin{bmatrix}S&\Hom_R(V,R)\\0&R\end{bmatrix}$$
are derived equivalent, where $\leftidx{_R}{V}{_S}$ is quasi-isomorphic to $\leftidx{_R}{M}{_S}$ and $\leftidx{_R}{V}$ is K-projective.  

Since $R$ is self dual we have that 
$$\Hom_R(V,R)\cong\Hom_R(V,DR)=\Hom_R(V,\Hom_k(R,k))$$
$$\cong\Hom_k(R\otimes_R V,k)\cong\Hom_k(V,k)=DV.$$

Furthermore, applying the functor $D(-)$ to the quasi-isomorphism $V\rightarrow M$ gives us the quasi-isomorphim $DM\rightarrow DV$. This in turn allows us to define a quasi-isomorphism $\begin{bmatrix}S&DM\\0&R\end{bmatrix}\rightarrow\begin{bmatrix}S&DV\\0&R\end{bmatrix}$, so
$$\begin{bmatrix}S&DM\\0&R\end{bmatrix} \textrm{ and }\begin{bmatrix}S&DV\\0&R\end{bmatrix}$$
are derived equivalent and hence 
$$\begin{bmatrix}R&M\\0&S\end{bmatrix} \textrm{ and }\begin{bmatrix}S&DM\\0&R\end{bmatrix}$$
are derived equivalent.
\end{proof}

\end{document}